\documentclass[letterpaper,11pt,oneside]{amsart}
\usepackage{amsmath}\usepackage{amssymb}\usepackage{latexsym}\usepackage{amsfonts}
\usepackage[utf8]{inputenc}\usepackage[english]{babel}\usepackage{multicol}
\usepackage{oldgerm}\usepackage{enumerate}

\newtheorem{theorem}{Theorem}  
\newtheorem{corollary}{Corollary}  \newtheorem{example}{Example}
\newtheorem{lemma}{Lemma} \newtheorem{proposition}{Proposition}

\newcommand{\mA}{\mathcal{A}} \newcommand{\mB}{\mathcal{B}}
\newcommand{\cF}{\mathbb{F}} \newcommand{\cZ}{\mathbb{Z}} 

\begin{document}

\title{Construction of $B_h[g]$ sets in product of groups}

\author{Diego Ruiz and Carlos Trujillo}
\noindent\address{Departamento de Matemáticas -- Universidad del Cauca\newline Calle 5 No. 4--70 -- Popayán, Colombia.\newline Grupo de Investigación: Álgebra, Teoría de Números y Aplicaciones, ERM -- ALTENUA.}
\email{dfruiz@unicauca.edu.co}
\email{trujillo@unicauca.edu.co}

\maketitle
\vspace{-0.7cm}
\centerline{\textsc{Department of Mathematics}}
\centerline{\textsc{Universidad del Cauca}}
\centerline{\textsc{Popayán--Colombia}}
\centerline{\small{dfruiz@unicauca.edu.co, trujillo@unicauca.edu.co}}
\vspace{0.3cm}

\begin{abstract}
A subset $\mA$ of an abelian group $G$ is a $B_h[g]$ set on $G$ if the elements of $G$ can be written in at most $g$ ways as sum of $h$ elements of $\mA$. Given any field $\cF$, this work presents constructions of $B_h[g]$ sets on the abelian groups $\left(\cF^h,+\right)$, $\left(\cZ^d,+\right)$, and $\left(\cZ_{m_1}\times\cdots\times\cZ_{m_d},+\right)$, for $d\geq2$, $h\geq2$, and $g\geq1$.
\medskip



\vspace{0.5cm}

\end{abstract}


\section{Introduction}
Let $g$ and $h$ denote positive integers with $h\geq2$. Let $G$ be an abelian additive group denoted by $(G,+)$. The set $\mA=\{a_1,\ldots,a_k\}\subseteq G$ is a $B_h[g]$ set on $G$ if every element of $G$ can be written in at most $g$ ways as sum of $h$ elements in $\mA$, that is, if given $x\in G$, the solutions of the equation $x=a_1+\cdots+a_h$, with $a_1,\ldots,a_h\in\mA$, are at most $g$ (up to rearrangement of summands). If $g=1$, $\mA$ is a $B_h$ set, while if $g=1$ and $h=2$, $\mA$ is a Sidon set.

Let $F_h(G,g)$ denote the largest cardinality of a $B_h[g]$ on $G$. If $g=1$ we write $F_h(G)$. Furthermore, if $G$ is the direct product of $d~\geq2$ abelian groups and $\mA$ is a $B_h[g]$ set on $G$, sometimes we say that $\mA$ is a $d-$dimensional $B_h[g]$ set on $G$. For $N\in\mathbb{N}$, let $[0,N-1]:=\{0,1,\ldots,N-1\}$. If $\cZ^d$ denotes the set of all $d-$tuples of integer numbers and $[0,N-1]^d$ denotes the cartesian product of $[0,N-1]$ with itself $d$ times, we define
\begin{equation*}
    F_h^d(N,g):=\max\{|\mA|:\mA\subseteq[0,N-1]^d,\,\mA\in B_h[g]\}.
    \label{funcionMaximal}
\end{equation*}

The main problem on $B_h[g]$ sets consists on establishing the largest cardinality of a $B_h[g]$ set on a finite group $G$. With analytical constructions it is possible to characterize lower bounds for $F_h(G,g)$, while using counting and combinatorial techniques, it is possible to characterize upper bounds. In this work we focus on constructions to obtain known lower bounds for $F_h(G,g)$ on particular groups $G$ ($\left(\cF^h,+\right)$, $\left(\cZ^d,+\right)$, and $\left(\cZ_{m_1}\times\cdots\times\cZ_{m_d},+\right)$ for any field $\cF$ and $d\geq2$, $h\geq2$, $g\geq1$), while other works are focused on upper bounds \cite{RuizBravoTrujillo}, \cite{CillerueloNd}, \cite{HigherDimensions}.

Different works have introduced constructions of $B_h[g]$ sets for particular values of $h$, and $g$. On $(\cZ,+)$, the most obvious construction of Sidon sets is given by Mian--Chowla using the greedy algorithm \cite{MianChowla}. This result is generalized by O'Bryant for any $h\geq 2$ and any $g\geq 1$ in \cite{OBryant}.

Other constructions of $B_h$ sets are due to Rusza, Bose, Singer, and Erd\"{o}s \& Tur\'{a}n. Rusza constructs a Sidon set on the group $\left(\mathbb{Z}_{ (p^2-p)},+\right)$ for $p$ prime. Bose's construction initially consider $h=2$ but could be generalized for any $h\geq2$ and any prime power $q$ on the group $\left(\mathbb{Z}_{q^h-1},+\right)$. Similarly to Bose, Singer constructs a $B_h$ set with $q+1$ elements on $\left(\mathbb{Z}_{(q^{h+1}-1)/(q-1)},+\right)$. Actually this construction can be established using Bose's construction \cite{TrujilloAlexis}. Finally, based on quadratic residues modulo a fixed prime $p$, Erd\"{o}s \& Tur\'{a}n construct Sidon sets on $(\mathbb{Z},+)$ \cite{OBryant}.

In dimension $d=2$ some constructions are due to Welch, Lempel, Golomb~\cite{Golomb}, Trujillo \cite{TrujilloDoc}, and C. Gómez \& Trujillo \cite{TrujilloAlexis}. Welch constructs Sidon sets with $p-1$ elements on the groups $\left(\cZ_{p-1}\times\cZ_p,+\right)$, $\left(\cZ_{p}\times\cZ_{p-1},+\right)$, generalized in \cite{UIS} to the groups $\left(\cZ_{q-1}\times \cF_q,+\right)$ and $\left(\cF_q\times \cZ_{q-1},+\right)$, respectively, where $\mathbb{F}_q$ is the finite field with $q$ elements. Golomb constructs Sidon sets with $q-2$ elements on the group $\left(\mathbb{Z}_{q-1}\times \mathbb{Z}_{q-1},+\right)$ (Lempel's construction is a particular case of Golomb). Trujillo in \cite{TrujilloDoc} presents an algorithm to construct Sidon sets on $\left(\mathbb{Z}\times \mathbb{Z},+\right)$ from a given Sidon set on $\left(\mathbb{Z},+\right)$. Finally, C. Gómez \& Trujillo construct $B_h$ sets on $\left(\mathbb{Z}_p\times \mathbb{Z}_{p^{h-1}-1},+\right)$ \cite{TrujilloAlexis}.

In higher dimensions, Cilleruelo in \cite{CillerueloNd} presents a way of mapping Sidon sets in $\mathbb{N}$ to Sidon sets in $\mathbb{N}^d$ for $d\geq2$, from which is possible to obtain a relation between the functions $F_h(N^d)$ and $F_h^d(N)$.

In this work we present constructions of $d-$dimensional $B_h[g]$ sets ($d\geq2$) on special abelian groups. The first construction uses the elementary symmetric polynomials and the Newton's identities to generalize a construction done initially for $d=2$ \cite{CillerueloDiablo}. In the second construction we generalize Trujillos's algorithm given in \cite{TrujilloDoc} to any dimension $d$ and all $h\geq 2,\,g\geq 1$, obtaining lower bounds for $F_h^d(N,g)$ from a known lower bounds for $F_h(N^d,g)$. Finally, using a homomorphism between abelian groups, we construct $d-$dimensional $B_h[g^\prime]$ sets from $d-$dimensional $B_h[g]$ sets, with $g$ a divisor of $g^\prime$.

\noindent The remainder of this work is organized as follows: For any finite field $\cF$, Section~\ref{Cons_xh} describes a construction of $B_h$ sets on $\left(\cF^h,+\right)$, where $\cF^h$ denotes the set of all $h-$tuples of elements of $\cF$. Section~\ref{ConsTrujillo} presents a construction of $B_h[g]$ sets on $\left(\cZ^d,+\right)$, and in Section~\ref{ConsNuestra} we construct $B_h[g]$ sets on $\left(\cZ_{m_1}\times\cdots\times\cZ_{m_d},+\right)$. Furthermore, we present a generalization of a Golomb Costas array construction. Finally, Section 4 describes the concluding remarks of this work. 

\section{Construction of $B_h$ sets on $\left(\cF^h,+\right)$}
\label{Cons_xh}
Let $p$ be a prime number. Note that $\mA:=\{(x,x^2):x\in\mathbb{Z}_p\}$ is a $B_2$ set on $\left(\mathbb{Z}_p\times\mathbb{Z}_{p},+\right)$ \cite{CillerueloDiablo}. In this section we generalize this construction using $h-$tuples ($h>2$). First we introduce the following notations and definitions.

Let $n$ be a positive integer. The elementary symmetric polynomials in the variables $x_1,\ldots,x_n$, written by $\sigma_k(x_1,\ldots,x_n)$ for $k=1,\ldots,n$, is defined as
\begin{equation*}
    \sigma_k(x_1,\ldots,x_n):=\sum\limits_{1\leq j_{1}<\cdots< j_{k}\leq n}x_{j_{1}}\cdots x_{j_{n}}.
    \label{EqElemSimePol}
\end{equation*}
If $k=0$ we consider $\sigma_0(x_1,\ldots,x_n)=1$. For $n=3$ we have
\begin{align*}
    \sigma_0(x_1,x_2,x_3)&=1,\\
    \sigma_1(x_1,x_2,x_3)&=x_1+x_2+x_3,\\
    \sigma_2(x_1,x_2,x_3)&=x_1x_2+x_1x_3+x_2x_3,\\
    \sigma_3(x_1,x_2,x_3)&=x_1x_2x_3.
\end{align*}
Note that the elementary symmetric polynomials appear in the expansion of a linear factorization of a monic polynomial
\[
    \prod_{j=1}^{n}(\lambda-x_j)=\sum_{k=0}^n (-1)^k\sigma_k(x_1,\ldots,x_n)\lambda^{n-k}.
\]
Note also that if $p_k(x_1,\ldots,x_n)=x_1^k+\cdots+x_n^k$, the Newton's identities are given by
\begin{equation}
    k\sigma_k(x_1,\ldots,x_n)=\sum_{i=1}^k(-1)^{i-1}\sigma_{k-i}(x_1,\ldots,x_n)p_i(x_1,\ldots,x_n),
    \label{EqNewton}
\end{equation}
for each $1\leq k\leq n$ and for an arbitrary number $n$ of variables.
\begin{theorem}\label{TheCons_xh}
Let $\cF$ be a field with characteristic zero or $p>h$. The set
\[
    \mA:=\{(x,x^2,\ldots,x^h):x\in \cF\},
\]
is a $B_h$ set on $\left(\cF^h,+\right)$.
\end{theorem}
\begin{proof}
    Let $s\in \mathbb{F}^h$. Suppose there exist two different representations of $s$ as sum of $h$ elements of $\mA$ as follows
    \[
        s=(a_1,\ldots,a_1^h)+\cdots+(a_h,\ldots,a_h^h)=(b_1,\ldots,b_1^h)+\cdots+(b_h,\ldots,b_h^h),
    \]
     $a_i,b_i\in\mathbb{F}$ for $i=1,\ldots,h$. Note that for all $k=1,\ldots,h$, $\sum_{i=1}^h a_i^k=\sum_{i=1}^h b_i^k$. Because $p_k(a_1,\ldots,a_h)=\sum_{i=1}^h a_i^k$ and $p_k(b_1,\ldots,b_h)=\sum_{i=1}^h b_i^k$, using \eqref{EqNewton} recursively we have $\sigma_i(a_1,\ldots,a_h)=\sigma_i(b_1,\ldots,b_n)$, for all $i=1,\ldots,h$, that is
    \begin{align*}
        a_1+\cdots+a_h&=b_1+\cdots+b_h,\\
        a_1a_2+\cdots+a_{h-1}a_h&=b_1b_2+\cdots+b_{h-1}b_h,\\
        &\cdots\\
        a_1\ldots a_h&=b_1\ldots b_h,
    \end{align*}
    which implies that the elements of the sets $\{a_1,\ldots,a_h\}$ and $\{b_1,\ldots,b_h\}$ are roots of the same polynomial $q(x)$ on $\mathbb{F}[x]$, i.e.,
    \[
        q(x)=(x-a_1)\cdots (x-a_h)=(x-b_1)\cdots (x-b_h).
    \]
     That is, $\{a_1,\ldots,a_h\}=\{b_1,\ldots,b_h\}$ ($\mathbb{F}[x]$ is a unique factorization domain). Thus, cannot be possible to have two different representations of $s\in\mathbb{F}$ as sum of $h$ elements of $\mathbb{F}^h$ and $\mA$ is a $B_h$ set on $\left(\cF^h,+\right)$.
\end{proof}
Consider the case when $\cF$ is the finite field $\cF_q$, with $q=p^n$ for some $n\in\mathbb{N}$ and $p$ prime. Note that the groups $\left(\mathbb{F}_{p^n},+\right)$ and $\left(\mathbb{F}_{p}^n,+\right)$ are isomorphic, because if $\theta$ is a root of an irreducible polynomial of degree $n$ over $\cF_{p}$ in an extension field, the function
\begin{equation}
    \begin{array}[c]{cccc}
        \phi: & \mathbb{F}_{p^{n}} & \rightarrow & \mathbb{F}_{p}^{n}\label{IsomoFandF}\\
        & a_{0}+\cdots+a_{n-1}\theta^{n-1} & \mapsto &(a_{0},\ldots,a_{n-1})
    \end{array}
\end{equation}
defines an isomorphism between them.
\begin{corollary}\label{Cor_xh}
    For all $p>h$ prime and for all $n\in\mathbb{N}$ there exists a $B_h$ set with $p^n$ elements on $\left(\mathbb{Z}_p^{hn},+\right)$.
\end{corollary}
\begin{proof}
It follows immediately from Theorem~\ref{TheCons_xh} and the isomorphism $\phi$ given in \eqref{IsomoFandF}.
\end{proof}
We illustrate these results in the following example.
\begin{example}
    Consider $h=n=2$ and $p=3$. Let $p(x)=x^2+1$ be an irreducible polynomial on $\mathbb{Z}_3$. Suppose that $\theta$ is a root of $p(x)$ in an extension field of $\mathbb{Z}_3$. The field with 9 elements is given by
    \begin{align*}
        \mathbb{F}_9& =\{a+b\theta:a,b\in\mathbb{Z}_3\}\\
        & =\{0,1,2,\theta,\theta+1,\theta+2,2\theta,2\theta+1,2\theta+2\}.
    \end{align*}
    Using Theorem~\ref{TheCons_xh} we know that
    \[
        \mA=\left\{
            \begin{array}[l]{l}
                (0,0),(1,1),(2,1),(\theta,2),(\theta+1,2\theta),(\theta+2,\theta),\\
                (2\theta,2),(2\theta+1,\theta),(2\theta+2,2\theta)
            \end{array}
            \right\}
    \]
    is a Sidon set on $\left(\mathbb{F}_9\times \mathbb{F}_9,+\right)=\left(\mathbb{F}_9^2,+\right)$. Furthermore, using Corollary~\ref{Cor_xh} we have
    \[
        \mB=\left\{
            \begin{array}[l]{l}
                (0,0,0,0),(0,1,0,1),(0,2,0,1),(1,0,0,2),(1,1,2,0),\\
                (1,2,1,0),(2,0,0,2),(2,1,1,0),(2,2,2,0)
            \end{array}
            \right\}
    \]
    is a Sidon set on $\left(\mathbb{Z}_3\times\mathbb{Z}_3\times\mathbb{Z}_3\times\mathbb{Z}_3,+\right)=\left(\mathbb{Z}_3^4,+\right)$.
\end{example}

\section{Construction of $B_h[g]$ sets on $\left(\cZ^d,+\right)$}
\label{SectionConstruction}
In this section we present a construction of $B_h[g]$ sets for all $h,g\geq2$ on $\left(\mathbb{Z}^d,+\right)$. This construction generalizes a construction introduced by Trujillo in \cite{TrujilloDoc} which allows to obtain Sidon sets on $\left(\mathbb{Z}\times\mathbb{Z},+\right)$ from a Sidon set on $\left(\mathbb{Z},+\right)$. Our generalization also allows to construct $d-$dimensional $B_h[g]$ sets for all $h,g\geq2$ and any dimension $d$, from which it is possible to determine a way to map $B_h[g]$ sets on $(\cZ,+)$ into $B_h[g]$ sets on $\left(\cZ^d,+\right)$. 

Let $d,N$ be positive integers greater than $1$. Let $\mA$ denote a subset of $\mathbb{Z}^+$. If $a\in\mA$, $[a]_N=(n_k,\ldots,n_1,n_0)_N$ represents the integer $a=n_kN^k+\cdots+n_1N+n_0$ in base $N$ notation, where $k$ is a nonnegative integer and $0\leq n_j\leq N-1$, for $j=0,1,\ldots,k$. We denote the set obtained from the representation of each element of $\mA$ in base $N$ as $[\mA]_N$. Because every positive integer can be written uniquely in base $N$, then
\begin{equation*}
    |\mA|=|[\mA]_N|.
    \label{CardEqu}
\end{equation*}
Note that if $\mA\subseteq\left[0,N^{d}-1\right]$, then 
$[\mA]_{N}\subseteq[0,N-1]^{d}$.

\begin{theorem}\label{Trujilloconst}
If $\mA$ is a $B_h[g]$ set contained in $[0,N^d-1]$, then $[\mA]_N$ is a $B_h[g]$ set contained in $[0,N-1]^d$.
\end{theorem}

\begin{proof}
Let $s$ be a $d-$tuple in $\mathbb{Z}^d$ obtained as sum of $h$ elements in $[\mA]_N$. Suppose there exist $g+1$ representations of $s$ as follows
\begin{equation}
    s=[a_{1,1}]_N+\cdots+[a_{1,h}]_N=\cdots=[a_{g+1,1}]_N+\cdots+[a_{g+1,h}]_N,
    \label{hsumandos}
\end{equation}
where $a_{i,j}\in\mA$ for all $1\leq i\leq g+1$, $1\leq j\leq h$. Consider the representation of each $a_{i,j}\in\mA$ in base $N$ as    $[a_{i,j}]_N=\left(n_{(d-1,i,j)},\ldots,n_{(0,i,j)}\right)$. Note that for any $1\leq i\leq g+1$
\begin{align*}
    [a_{i,1}]_N+\cdots+[a_{i,h}]_N =&\left(n_{(d-1,i,1)},\ldots,n_{(0,i,1)}\right)+\cdots+\left(n_{(d-1,i,h)},\ldots,n_{(0,i,h)}\right)\\
    =&\left(n_{(d-1,i,1)}+\cdots+n_{(d-1,i,h)},\ldots,n_{(0,i,1)}+\cdots+n_{(0,i,h)}\right).
\end{align*}
Furthermore
\[
    \left(n_{(d-1,i,1)}+\cdots+n_{(d-1,i,h)}\right)N^{d-1}+\cdots+\left(n_{(0,i,1)}+\cdots+n_{(0,i,h)}\right)=a_{i,1}+\cdots+a_{i,h}
\]
which implies from \eqref{hsumandos} that
\begin{equation}
    a_{1,1}+\cdots+a_{1,h}=\cdots=a_{g+1,1}+\cdots+a_{g+1,h}.
    \label{hsumasenA}
\end{equation}
Because $\mA$\ is a $B_h[g]$ set, using \eqref{hsumasenA} we know there exist $\ell,m$ with $\ell\neq m$ and $1\leq \ell,m\leq g+1$, such that
\[
    \{a_{\ell,1},\ldots,a_{\ell,h}\}=\{a_{m,1},\ldots,a_{m,h}\}.
\]
Since representation in base $N$ notation is unique we have
\[
    \{[a_{\ell,1}]_{N},\ldots,[a_{\ell,h}]_{N}\}=\{[a_{m,1}]_{N},\ldots,[a_{m,h}]_{N}\},
\]
That is, it is not possible to have $g+1$ representations of $s$ as sum of $h$ elements of $\mA$. Therefore $[\mA]_{N}$ is a $B_h[g]$ set contained in $[0,N-1]^d\subset\left(\cZ^d,+\right)$.
\end{proof}
\begin{example}\label{Ejemplo1}
Note that $\mA=\{1,2,7\}$ is a Sidon set on $\left(\mathbb{Z}_8,+\right)$. In \cite{TrujilloDoc} Trujillo constructs a $B_2[2]$ set on $(\cZ,+)$ as follows
\[
    \mathcal{B}:=\mA\cup\left(\mA+m\right)\cup\left(\mA+3m\right)=\{1,2,7,9,10,15,25,26,31\},
\]
with $m=8$. Because $\mB\subseteq[0,2^5-1]$,
using Theorem \ref{Trujilloconst} we have that
\[
    [\mathcal{B}]_{2}=
    \left\{
        \begin{array}[c]{l}
            (0,0,0,0,1),(0,0,0,1,0),(0,0,1,1,1),(0,1,0,0,1),(0,1,0,1,0),\\
            (0,1,1,1,1),(1,1,0,0,1),(1,1,0,1,0),(1,1,1,1,1)
        \end{array}
    \right\}
\]
is a $B_{2}[2]$ set contained in $[0,1]^5$. 
Note also that $\mathcal{B}\subseteq[0,6^{2}-1]$, so
\[
    [\mathcal{B}]_6=\{(0,1),(0,2),(1,1),(1,3),(1,4),(2,3),(4,1),(4,2),(5,1)\}
\]
is a $B_{2}[2]$ set contained in $[0,5]^{2}$. 
\end{example} 
\section{Construction of $B_h[g]$ sets on $\left(\cZ_{m_1}\times\cdots\times\cZ_{m_d},+\right)$}
    \label{ConsNuestra}
\label{SecModular}
This section extend a construction of $B_h[g]$ sets given in \cite{TrujilloIndia} for $h=2$ and $d=1$, to all $h\geq2$ and any dimension $d>1$. First, we introduce the following result.
\begin{lemma}\label{Jhonny}
    Let $G$ and $G^\prime$ be two abelian groups and let $\phi:G\rightarrow G^\prime$ define a homomorphism. If $\mA$ is a $B_h[g]$ set on $G$ and $|Ker(\phi)|=g^\prime$, then $\phi(\mA)$ is a $B_h[gg^\prime]$ set on $\phi(G)$, where $gg^\prime$ denotes the product between $g$ and $g^\prime$.
\end{lemma}
The proof is given in \cite{EstTrujillo}.

Now, let $m_1,\ldots,m_d$ and $g_1,\ldots,g_d$ be positive integers. Using Lemma~\ref{Jhonny} we have the following result.
\begin{theorem}\label{ConsModular}
    Let $\mathcal{A}$ be a $B_h[g]$ set on $\left(\mathbb{Z}_{m_1}\times\cdots\times\mathbb{Z}_{m_d},+\right)$. If $g_1,\ldots,g_d$ are divisors of $m_1,\ldots,m_d$, respectively, then
    \[
        \mB:=\left\{\left(a_1\bmod \frac{m_1}{g_1},\ldots,a_d\bmod \frac{m_d}{g_d}\right):(a_1,\ldots,a_d)\in\mathcal{A}\right\}
    \]
    is a $B_h[gg_1\cdots g_d]$ set on $\left(\mathbb{Z}_{\frac{m_1}{g_1}}\times\cdots\times\mathbb{Z}_{\frac{m_d}{g_d}},+\right)$.
\end{theorem}
\begin{proof}
    Using notation used in Lemma~\ref{Jhonny}, let $G=(\mathbb{Z}_{m_1}\times\cdots\times\mathbb{Z}_{m_d},+)$ and $G^\prime=~\left(\mathbb{Z}_{\frac{m_1}{g_1}}\times\cdots\times\mathbb{Z}_{\frac{m_d}{g_d}},+\right)$ and define the homomorphism $\phi:G\rightarrow G^\prime$ as $\phi(b_1,\ldots,b_d)=\left(b_1\bmod \frac{m_1}{g_1},\ldots,b_d\bmod \frac{m_d}{g_d}\right)$. We establish $Ker(\phi)$ as follows. Note that $(b_1,\ldots,b_n)\in Ker(\phi)$ if and only if $\phi(b_1,\ldots,b_n)=(0,\ldots,0)$, that is, if
    \[
        \left(b_1\bmod \frac{m_1}{g_1},\ldots,b_d\bmod \frac{m_d}{g_d}\right)=(0,\ldots,0).
    \]
     Note also that $b_i\bmod \frac{m_i}{g_i}=0$ if and only if $b_i=k_i\frac{m_i}{g_i}$, for $k_i\in [1,g_i]$ and for all $i=1,\ldots,d$, which implies that $b_i\bmod \frac{m_i}{g_i}=0$ in exactly $g_i$ values. Thus, $|Ker(\phi)|=\prod_{i=1}^d g_i$. Finally, using Lemma~\ref{Jhonny} we have that $\mB=\phi(\mA)$ is a $B_h[gg_1\cdots g_d]$ set on $\left(\mathbb{Z}_{\frac{m_1}{g_1}}\times\cdots\times\mathbb{Z}_{\frac{m_d}{g_d}},+\right)$.
\end{proof}
Given $q$ a prime power and $\mathbb{F}$ a field, to illustrate Theorem~\ref{ConsModular} we present a construction of Sidon sets on $\left(\mathbb{Z}_{q-1}\times\mathbb{Z}_{q-1},+\right)$, which is based on the discrete logarithm\footnote{If $\theta$ is a primitive of $\mathbb{F}_q$, $\log_\theta (x)$ denotes the unique integer $k\in[1,q-1]$ such that $\theta^k=x$ on $\mathbb{F}_q$.} on $\mathbb{F}_q$.
\begin{proposition}\label{PropGolomb}
    Let $q=p^n$ a prime power. If $\alpha,\beta$ are primitive elements of $\mathbb{F}_q^*$ and $a\in\mathbb{F}_q^*$, then
    \begin{equation}
        \mathcal{G}(\alpha,\beta,a):=\{(i,\log_\beta(a-\alpha^i)):i=1,\ldots,q-1,\,\alpha^i\neq a\}
        \label{EqPropGolomb}
    \end{equation}
    is a Sidon set on $\left(\mathbb{Z}_{q-1}\times\mathbb{Z}_{q-1},+\right)$.
\end{proposition}
\begin{proof}
    Suppose there exist $u,v,w,y\in\mathcal{G}(\alpha,\beta,a)$ such that $u+v=w+y.$ Using~\eqref{EqPropGolomb} we know that there exist $i,j,k,\ell\in[1,q-1]$ such that
    \begin{equation}
        (i,\log_\beta(a-\alpha^i))+(j,\log_{\beta}(a-\alpha^{j})) =(k,\log_\beta(a-\alpha^{k}))+(\ell,\log_{\beta}(a-\alpha^{\ell}))
        \label{Eq1Golomb}
    \end{equation}
    where $\alpha^i,\alpha^j,\alpha^k,\alpha^\ell$ are not equal to $a$. From \eqref{Eq1Golomb} we have
    \begin{align*}
        (i+j)& \equiv(k+\ell)\bmod(q-1),\\
        \log_\beta(a-\alpha^i)+\log_\beta(a-\alpha^j)& \equiv(\log_\beta(a-\alpha^k)+\log_\beta(a-\alpha^\ell))\bmod(q-1),
    \end{align*}
    what implies that $(a-\alpha^i)(a-\alpha^j)=(a-\alpha^k)(a-\alpha^\ell)$.
    We have in $\mathbb{F}_q^*$
    \begin{align*}
        \alpha^i\alpha^j & =\alpha^k\alpha^\ell,\\
        \alpha^i+\alpha^j & =\alpha^k+\alpha^\ell,
    \end{align*}
    that is, $\alpha^i,\alpha^j$, and $\alpha^k,\alpha^\ell$ are roots of a polynomial $q(x)\in\mathbb{F}[x]$ of degree 2 (i.e., $q(x)=(x+\alpha^i)(x+\alpha^j)=(x+\alpha^k)(x+\alpha^\ell)$). Therefore, $\{\alpha^i,\alpha^j\}=\{\alpha^k,\alpha^\ell\}$ and $\{i,j\}=\{k,\ell\}$, which implies that is not possible to have two representations of an element in $\mathbb{Z}_{q-1}\times\mathbb{Z}_{q-1}$ as sum of two elements of $\mathcal{G}(\alpha,\beta,a)$. That is, $\mathcal{G}(\alpha,\beta,a)$ is a Sidon set on $(\mathbb{Z}_{q-1}\times\mathbb{Z}_{q-1},+)$.
\end{proof}
\begin{example}
First we apply Proposition~\ref{PropGolomb} to construct a Sidon set on $\langle\mathbb{Z}_{16}\times\mathbb{Z}_{16},+\rangle$. Let $q=p=17$, and let $\alpha=3,\,\beta=5$ be primitive elements of $\mathbb{Z}_{17}^*$. With $a=1$
\begin{align*}
    \mathcal{G}(3,5,1)=
    \left\{
        \begin{array}[c]{l}
            (1,14),(2,10),(3,2),(4,1),(5,4),(6,13),(7,15),(8,6),\\
            (9,12),(10,7),(11,11),(12,5),(13,3),(14,8),(15,9)
        \end{array}
    \right\}
\end{align*}
is a Sidon set on $(\mathbb{Z}_{16}\times\mathbb{Z}_{16},+)$. Now, if $g_1=g_2=2$, using Theorem~\ref{ConsModular},
\[
    \mA=\left\{
            \begin{array}[c]{l}
                (1, 6), (2, 2), (3, 2), (4, 1), (5, 4), (6, 5), (7, 7), (0, 6), \\
                (1, 4), (2, 7), (3, 3), (4, 5), (5, 3), (6, 0), (7, 1)
            \end{array}
          \right\}
\]
is a $B_2[4]$ set on $(\mathbb{Z}_{8}\times \mathbb{Z}_{8},+)$.
\end{example} 
\section{Concluding remarks}
Using the constructions given in this work we can obtain lower bounds and closed formulas for $F_h^d(G,g)$, for some abelian group $G$ and some values of $d,h$ and $g$.

Note from Theorem~\ref{TheCons_xh} and Corollary~\ref{Cor_xh} that $F_2^h(\cF^h_q)\geq q$ for $q$ a prime power. Particularly if $h=2$ and $q=p$ prime we have $F_2^2(\mathbb{Z}_{p}\times \mathbb{Z}_{p})\geq p$, but it is easy to establish that $F_2^2(\mathbb{Z}_{p}\times \mathbb{Z}_{p})=p$ \cite{UIS}. A natural question to state is the following: Can we obtain a similar result, as the last one, on the group $(\cZ_p\times\cZ_p\times\cZ_p,+)$? That is,
\[
    F_2^3(\cZ_p\times \cZ_p\times\cZ_p)\sim p^{3/2}?
\]
Now, using Theorem~\ref{Trujilloconst}, for integers $d,g,N\geq 1$ and $h\geq2$ we know that
\begin{equation*}
    F_h(N^d,g)\leq F_h^d(N,g)
    \label{Implication}
\end{equation*}
Particularly, if $d=2$, $h=2$, and $g=1$ we have that $F_2^1(N^2)\leq F_2^2(N)$, which implies that good constructions of Sidon sets on $\cZ$ give good lower bounds for Sidon sets on $\cZ\times\cZ$. Furthermore, an interesting work consists in to analyze the behavior of the difference $F_2^2(N)-F_2^1(N^2)$ when $N$ grows.

Finally, from Proposition~\ref{PropGolomb} we can establish that $F_2^2(\cZ_{q-1}\times\cZ_{q-1})\geq q-2$, which lead us to wonder if is it possible to state that $F_2^2(\cZ_{q-1}\times\cZ_{q-1})= q-1$?

\vspace{0.5cm}
\textbf{Acknowledgment.} We thank to COLCIENCIAS and Universidad del Cauca for the support to our research group ``Álgebra, Teoría de Números y Aplicaciones--ALTENUA ERM'' under the projects 110356935047 and VRI--3744. This work is dedicated to the memory of Javier Cilleruelo (1961-2016).


\end{document}